\documentclass[10pt]{article}
\usepackage{amssymb,amsfonts}
\usepackage{pb-diagram}
\usepackage{amsmath,amsthm,amsfonts,amssymb}
\usepackage[usenames]{color}
\usepackage{hyperref}
\usepackage{soul}
\usepackage{graphicx}
\usepackage{amssymb}
\usepackage{amsmath}
\usepackage{amssymb}
\usepackage{mathrsfs}

\newtheorem{theorem}{Theorem}

\newtheorem{lemma}{Lemma}
\newtheorem{cor}{Corollary}

\newtheorem{problem}{Problem}
\newcommand{\be}{\begin{enumerate}}
\newcommand{\ee}{\end{enumerate}}
\newcommand{\beq}{\begin{equation}}
\newcommand{\eeq}{\end{equation}}

\def\N{{\mathbb{N}}}
\def\Z{{\mathbb{Z}}}

{}
{}

\title{The Diophantine Problem in Some Metabelian  Groups }
\author{Olga Kharlampovich\footnote{Hunter College and Graduate Center, CUNY},  Laura L\'opez\footnote{Graduate Center, CUNY}, Alexei Myasnikov\footnote{Stevens Institute}}

\date{}

\begin{document}

\maketitle
\begin{abstract} In this paper we show that the Diophantine problem for quadratic equations in Baumslag-Solitar groups $BS(1,k)$ and in wreath products $A \wr \mathbb{Z}$, where $A$ is a finitely generated abelian group and   $\mathbb{Z}$ is an infinite cyclic group, is decidable. We show also that one can decide if there are  non-trivial solutions of systems of equations without coefficients in these groups and give some sort of description of solutions.  Previously we stated that there is an algorithm that given a finite system of equations with constants in such a group decides whether or not the system has a solution in the group, this proof, unfortunately, has a gap.

\end{abstract}
\section{Introduction}
The problem of solving equations in various classes of groups and monoids
has been an active research field for many years now. The first general results on equations in groups appeared in the 1960's in the works of Lyndon \cite{lyndon} and Malcev \cite{malcev}.  In the 1970's  Makanin  \cite{mak77,mak82} proved
the solvability of (systems of) equations  for free monoids and free
groups. Makanin's decidability results have been extended to hyperbolic groups
and right-angled Artin groups \cite{DM06}, and it was shown that certain group operations
(graph products \cite{DL04}, HNN-extensions and amalgamated products over finite groups)
preserve decidability \cite{LS}. Moreover, a significant progress concerning the computational
complexity and the structure of solution sets have been obtained in recent years.
On the negative side,  by the Ershov-Romanovskii-Noskov
result the first-order theory of a finitely generated solvable group is decidable if and
only if the group is virtually abelian. The corresponding problem has been posed in \cite{KRRR}.
Ershov proved this statement \cite{E} in the
nilpotent case,  Romanovskii \cite{Rom}  generalized it to the polycyclic case, and finally, Noskov  \cite{Noskov} established the
most general statement for the case of a finitely generated solvable group.  Denote by ${\mathcal E\mathcal P_1}$ the problem of  solvability of one equation. Roman'kov showed that ${\mathcal E\mathcal P_1}$ is undecidable even for the subclass of all split equations of the form $w(x_1,\ldots ,x_n)=g,$ where $w(x_1,\ldots ,x_n)$ is a coefficient-free word and $g$ is an element of the underlying group $G$ that is a free nilpotent of class $\geq 9$ \cite{R1} (this bound was later reduced to $\geq 4$ in \cite{R}) or $G$ is a free metabelian non-abelian group \cite{R}. In \cite{DLS} the authors proved that ${\mathcal E\mathcal P_1}$ is decidable in the Heisenberg group that is free nilpotent of rank 2 and class 2. But the Diophantine problem (denoted by ${\mathcal E\mathcal P}$ in \cite{DLS})
is undecidable in any non-abelian free nilpotent group.

In this paper we show that the Diophantine problem for quadratic equations in solvable Baumslag-Solitar groups $BS(1,k)$ and in wreath products $A \wr \mathbb{Z}$, where $A$ is a finitely generated abelian group and   $\mathbb{Z}$ is an infinite cyclic group, is decidable, i.e. there is an algorithm that given a finite  quadratic system of equations with constants in such a group decides whether or not the system has a solution in the group. We show also that one can decide if there are  non-trivial solutions of systems of equations without coefficients in these groups.  In the published version of this paper we stated that there is an algorithm that given a finite system of equations with constants in such a group decides whether or not the system has a solution in the group, this proof, unfortunately, has a gap.

The metabelian Baumslag-Solitar groups are defined by  one-relator presentations $BS(1,k)=\langle a,b \mid b^{-1}ab=a^k
\rangle$, where $k \in \mathbb{N}$. If $k = 1$ then $BS(1,1)$ is free abelian of rank $2$, so the Diophantine problem in this group is  decidable  (it reduces to solving finite systems of linear equations over the ring of integers $\mathbb{Z}$). Furthermore, the first-order theory of $BS(1,1)$ is also decidable \cite{Szmielew}.  However, if  $ k \geq 2$ then $BS(1,k)$ is metabelain which is not virtually abelian, so the first-order theory of $BS(1,k)$ is  undecidable by \cite{Noskov}.  As we mentioned above, in free metabelian non-abelian groups equations are undecidable \cite{R}. In fact,  in a finitely generated metabelian group $G$ given by a finite presentation in the  variety ${\mathcal M}_2$  of metabelian groups, the Diophantine problem is undecidable asymptotically almost surely if the deficiency of the presentation is at least 2    \cite{GMO}.

 In general,  if  the quotient $G/\gamma_3(G)$ of a finitely generated metabelian group $G$ by its third term of the lower central series is a non-virtually abelian nilpotent group, then the decidability of the Diophantine problem in $G$ would imply decidability of the Diophantine problem for some finitely generated ring of algebraic integers $O_G$ associated with $G/\gamma_3(G)$.  The latter seems unlikely, since there  is a well-known  conjecture in number theory (see, for example, \cite{DL,PZ}) that states that the Diophantine problem in rings of algebraic integers is undecidable.  The discussion above shows that finitely generated metabelian groups $G$ with virtually abelian quotients $G/\gamma_3(G)$ present an especially interesting case in the study of equations in metabelian groups. The groups $BS(1,k)$  and wreath products $A \wr \mathbb{Z}$, where $A$ is a finitely generated abelian group and   $\mathbb{Z}$ is an infinite cyclic group, are the typical examples of such groups.  
 This gives also a new look at one-relator groups.  The groups $BS(1,k)$, $k\geq 2$,  were until recently the only known examples of one-relator groups with undecidable first-order theory. Recently, we were able to show (still unpublished)  that any one-relator group containing non-abelian group $BS(1,k)$ has undecidable first-order theory. 
 However, it is quite possible that equations in such groups are still decidable.

\section{Equations in $BS(1,k)$}
Our first main result is \begin{theorem} \label{main_th}
Quadratic equations in  $BS(1,k)$ are decidable. There is also an algorithm to decide if there is a  non-trivial solution of a system of equations without coefficients.
\end{theorem}
To prove the theorem we have to construct an algorithm that decides whether the set of formulas of the form $\exists
\bar x \, \,  \wedge _{i=1}^s t_i(\bar x, a, b) = 1$ is decidable, where $t_i(\bar x, a, b)$ is a
group word in the alphabet $\bar x, a, b$. Recall that the group $BS(1,k)$ is isomorphic to the group $\mathbb{Z}[1/k]
\rtimes \mathbb{Z}$, where $\mathbb{Z}[1/k] \cong ncl(a)$ and
$\mathbb{Z} \cong \langle b \rangle$, where $$\mathbb{Z}[1/k] =\{zk^{-i}, z\in\mathbb Z, i\in\mathbb N\}$$ and the action of $\langle b \rangle$ is given by $b^{-1}ub=u^k$. Thus, we can think of elements in $BS(1,k)$
as pairs $(zk^{-i} , r)$ where $z,r,i \in \mathbb{Z}$.  The product is defined as
$$(z_1k^{-i_1} , r_1)(z_2k^{-i_2} , r_2)=(z_1k^{-i_1}+z_2k^{-(i_2+r_1)} , r_1+r_2).$$

The inverse of an element $(zk^{-y},r)$ is $(-zk^{-(y-r)}, -r)$
 
 The following lemma reduces systems of equations in $BS(1,k)$ to systems of equations in $\mathbb Z$.

\begin{lemma}
Any finite system of equations in $BS(1,k)$ is equivalent to a finite system of equations of the form
\begin{equation}\label{(1)}
\sum_i z_i k ^{-y _i}(\sum_j \pm k^{\tau_{ij}(\bar r)})-\sum_t \gamma _tk^{\tau_t(\bar r)}=0 
\end{equation}

and \begin{equation}\label{(2)}
\sum \beta_j r_j = \delta.
\end{equation}
where
$ \tau _t(\bar r), \tau_{ij}(\bar r) = \sum_q \alpha_q r_q +c_q$ and where $\alpha_q, c_q, \delta, \gamma _t, \beta_j \in \Z$, and $ y_i,z_i,r_i,$ are variables. 

The product  $z_i k ^{-y _i}$ can be also considered as one variable in $\mathbb{Z}[1/k].$ 
\end{lemma}

\begin{proof}
Note that

$$(z_1k^{-y_1}, r_1) \cdot (z_2k^{-y_2}, r_2) \cdots (z_nk^{-y_n}, r_n) = $$
$$(z_1k^{-y_1}+ z_2k^{-(y_2+r_1)} + ... + z_nk^{-(y_n+r_1+...+r_{n-1})},r_1+...+r_n)$$

The system of  equations in the first and second component corresponds to a system of equations of the form (\ref{(1)}) and (\ref{(2)}), respectively.

\end{proof}

To solve a system of equations in $BS(1,k)$, we begin by solving system (\ref{(2)}). This system is just a linear system of equations $AX=B$ with integer coefficients, where $X=(r_1,\ldots ,r_n)^T$ and $A$ is the matrix of the system. Using integral elementary column operations on $A$ and row operations on $(A|B)$ we can obtain an  equivalent system $\bar A\bar X=\bar B$ such that $\bar A$ has a diagonal form. This is Smith normal form. Column operations on $A$  correspond to change of variables. Row operations on $(A|B)$ correspond to transformations of the system of equations into an equivalent system.
If the system $\bar A\bar X=\bar B$ does not have a solution, then the corresponding system of equations in the group does not have a solution. If the system $\bar A\bar X=\bar B$  is solvable, then we change  variables $X$ to $\bar X$. Some of the new variables $\bar X$ will have fixed integer values and some will be arbitrary integers. Substitute those $\bar X$'s  into system (\ref{(1)}).   We only have to check that there exist integer solutions $Z=\{z_1,\ldots , z_n\}$, $Y=\{y_1,\ldots , y_n\}$ and remaining $\bar X$ that we denote $\hat X = \{r_{i_1} \ldots r_{i_m}\}$.

We say that a system of equations $S(X)=0$ with variables $X$ is equivalent to a disjunction of  systems $S_1(X)=0,\ldots ,S_m(X)=0$ if every solution of $S(X)=1$ is a solution of one of $S_i(X)=0, i=1,\ldots ,m$ and every solution of $S_i(X)=0$ is a solution of $S(X)=0$. One can consider  system (\ref{(1)}) as a linear system with  variables  $z_ik^{-y_i},$ and   linear combinations of exponential functions as coefficients (which contain variables $\hat X$). It can be transformed using row operations to an equivalent disjunction of triangular like systems (with respect to variables $z_s k ^{-y _s}$, $s=1,\ldots ,q$) of the following form:
 \begin{equation}\label{(3)}
z_s k ^{-y _s}(\sum_j \delta _{sj} k^{\tau_{sj}(\bar r)})=\sum _{i>q}z_i  k ^{-y _i}(\sum_j \delta _{ij} k^{\sigma_{ij}(\bar r)})+\sum_t \gamma _tk^{\tau_t(\bar r)}, s=1,\ldots ,q,
\end{equation}

 \begin{equation}\label{4}
\sum_j a_{j}k^{\phi_{j}(\bar r)}=0\  ({\rm system\  of\  such\ equations}). 
\end{equation}
where $\delta _{sj}, \delta _{ij}, \gamma _t, a_{j} \in \mathbb Z$ and $\tau_{sj}, \sigma_{ij}, \tau_t, \phi_{j}$  are linear combinations of elements in $\hat X$ and constants. 
We will get a disjunction of systems because when multiplying equations by some coefficient we have to consider separately the case when this coefficient is zero.
\newline 
Now we have to solve systems (\ref{(3)}) and (\ref{4}). We will first find all solutions of system (\ref{4}). Semenov's ideas in \cite{Semenov2} (where he
proved that the theory of $\langle \mathbb{Z}, +, k^x \rangle$ is decidable) can be used to prove the following lemma.

\begin{lemma} \label{lemma1-S} 
Any system of equations over $\Z$ of the form
 \begin{equation}\label{5}
F(\bar y)=\sum_j \beta _jk^{y_j}+C=0, 
\end{equation}
where $\beta_j\in \Z$, $k\in\mathbb N, k>1,$
 with  variables $\bar y
=(y_1,...,y_n)$, is equivalent to a disjunction of linear systems of equations over $\Z$.
\end{lemma}

\begin{proof}
 Let $\bar y = (y_1, \ldots ,y_n)$ and let $\lambda: \{y_1, \ldots ,y_n\} \rightarrow \{+, -\}$ be a map that assigns to each variable a positive or negative sign (the agreement will be that zero has a positive sign). System (\ref{5}) over $\mathbb{Z}$ is equivalent to a disjunction of $2^n$ systems each with an assignment $\lambda$. Now we fix one of these systems and we show how to describe all solutions. 
 
We begin by rewriting each equation so that all variables are positive. We may do this by substituting in each equation $-y_i$ for $y_i$ for each $y_i$ that has a negative assignment. Then we multiply each equation by $k^{y_{i_1}+\ldots+y_{i_s}}$, where $y_{i_1}, \ldots ,y_{i_s}$ are all the variables whose signs were changed. For instance, suppose we have an equation $k^{y_1} -k^{y_2} + k^{y_3} + c= 0$ with assignment $y_1 < 0, y_2 \geq 0, y_3 \geq 0$. Then we rewrite it as $k^{-y_1} - k^{y_2} + k^{y_3} +c = 0$ with assignment $y_1\geq 0,y_2\geq 0,y_3\geq 0$ and multiply the equation by $k^{y_1}$. We then obtain the equation
$$1- k^{y_1 + y_2} + k^{y_1 + y_3}+ck^{y_1}=0$$ with assignment $y_1\geq 0,y_2\geq 0,y_3\geq 0$.
We now obtain a system over $\N$ of the form 
$$\sum_i \beta _ik^{\sum_j y_{ij}}+C=0$$

Next, we substitute all sums in exponents of $k$ by new variables to obtain a system of equations over $\N$ of the form 
\begin{equation} \label{eq-6}
F'(\bar y) = \sum_i \beta _ik^{\hat y_i}+C=0
\end{equation} 
\newline
\textbf{Claim:} A finite system  of equations in the form (\ref{eq-6}) is equivalent  to a disjunction of  systems of linear equations of the form $\{\hat y_1 = \hat y_2 + c_1, \hat y_2 = \hat y_3 + c_2, \ldots, \hat y_{s-1} = \hat y_s + c_s\}$. 
\newline
\begin{proof}
Denote the new variables as $\bar y'=(\hat y_1, \ldots \hat y_m)$.
We begin by showing that for each $i$, there is a $\Delta_i \in \mathbb{N}$ such that system (\ref{eq-6}) does not have a solution if $\hat y_i > \hat y_j + \Delta_i$ for all $j \neq i$.  

Fix $i$. We can rewrite each equation in the system in the form $k^{\hat y} +\sum_i \gamma _ik^{\hat x_i} = \sum_j \delta _jk^{\hat z_j}+C$, where all $\gamma _i, \delta _j$ are positive integers, $\hat y=\hat y_i$ and $\hat x_i, \hat z_j$ are all variables in $\bar y' - \hat y_i$. For each equation, let $\Delta > \log _k(\sum_j \delta_j+C)$ if  $C\geq 0$ and $\Delta > \log _k(\sum_j \delta_j)$ if $C<0$, and $\hat y > \hat x_i + \Delta$ and $\hat y > \hat z_j + \Delta$ for all $i, j$. Then $k^{\hat y}>k^{\Delta} k^{\hat z_j} > (\sum_j \delta _j+C) k^{\hat z_j}$ for all $j$. Thus, the right side of the equation will always be smaller than the left side, and  the equation has no solution. Thus, we can take $\Delta_i$ to be the smallest such $\Delta$.

So we have shown that for all variables $\hat y_i$, if $F'$ (or a finite system of equations where each equation has form $F'$) has a solution then there is a $j \neq i$ such that $\hat y_i \leq \hat y_j + \Delta_i$.  Now consider a finite graph $\mathcal{G}$ with $n$ vertices labeled $\hat y_1, \ldots ,\hat y_m$ and directed edges from $\hat y_i$ to $\hat y_j$ whenever $\hat y_i \leq \hat y_j + \Delta_i$. Note that each vertex must be the initial vertex of some edge and thus the graph must contain a cycle in every connected component. Suppose there is a cycle $\hat y_{i_1}, \ldots ,\hat y_{i_s}=\hat y_{i_1}, \, s \leq m+1$.  Then $$\hat y_{i_1} \leq \hat y_{i_2} + \Delta_{i_1} \leq \hat y_{i_3} + \Delta_{i_2} + \Delta_{i_1} \leq \ldots \leq \hat y_{i_s} + \Delta_{i_{(s-1)}} + \ldots + \Delta_{i_1}$$ $$= \hat y_{i_1} + \Delta_{i_{(s-1)}} + \ldots + \Delta_{i_1}$$
Therefore for any $2 \leq j \leq s-1$, we have that $$\hat y_{i_1} - \sum_{t=1}^{j-1} \Delta_{i_t} \leq \hat y_{i_j} \leq \hat y_{i_1} + \sum_{t=j}^{s-1} \Delta_{i_t}$$
Therefore, the value of any $\hat y_{i_j}$ with $2 \leq j \leq s-1$ is bounded by the value of $\hat y_{i_1}$. 

Fix a $y_{i_j}$ and let $\Delta_{j_1}= \sum_{t=1}^{j-1} \Delta_{i_t}$ and $\Delta_{j_2}= \sum_{t=j}^{m-1} \Delta_{i_t}$. Then we may replace the equation $F'(\bar y)$ by a disjunction of equations $G(\bar y \backslash \hat y_{i_j})$ where $G$ is the same as the formula $F'$, but $\hat y_{i_j}$ is replaced by $\hat y_{i_1} -\Delta_{j_1}$ in one equation, $y_{i_1} -\Delta_{j_1}+1$ in the next, and so on until $y_{i_1} +\Delta_{j_2}$.

Now we may eliminate variables from each equation in $m$ variables inductively, obtaining at each step a new disjunction consisting of  a system of equations in less variables and a set of linear equations of the form $\hat y_i = \hat y_j + c_i$ which we use to eliminate one variable. At the last level of each branch of this procedure, we will have one of three possible outcomes: 
\begin{enumerate} 

\item All exponential terms have canceled out and we have a false equation with constant terms. In this case there is no solution to (\ref{eq-6}) or (\ref{5}) in this branch.
\item There is an equation $0=0$ (i.e. all terms cancel out after a substitution). In this case all variables (after renumbering) $ \hat y_{i+1}, \ldots, \hat y_m$ that remained in the previous step of the branch are taken as free variables, and we obtain a general solution $\hat y_1=\hat y_2 + c_1, \hat y_2 = \hat y_3 + c_2, \ldots ,\hat y_i=\hat y_{i+1} + c_i$ to system (\ref{eq-6}) along this branch.
 \item There is one equation left of the form $\beta_s k^{y_s} + C =0$. In this case, this equation has a unique solution $y_s=b$ or no solution.
\end{enumerate} 
In the second case, any solution in $\Z$ of the linear system $\hat y_1=\hat y_2 + c_1, \hat y_2 = \hat y_3 + c_2, \ldots ,\hat y_i=\hat y_{i+1} + c_i$  will be a solution to system (\ref{eq-6}) since when we substitute the variables into this equation, the same cancellations will occur and we will remain with the equation $0=0$.
This proves the claim.
\end{proof}
System (\ref{5}) can also be reduced to a disjunction of linear systems by substituting each $\hat y_i$ back to the corresponding linear combination of $y_1, \ldots, y_n$. This completes the proof of the lemma.
\end{proof}

System (\ref{4}) is also equivalent to a disjunction of linear systems --we first replace sums appearing in the exponent of $k$ by new variables and then apply Lemma \ref{lemma1-S}.
We now solve this disjunction of linear systems --if it is solvable, the general solution will correspond to the disjunction of  systems of linear equations on $\hat X$.  We fix one of these systems and substitute those $r_i$'s that  are fixed numbers into system (\ref{(3)}) that has triangular form.  Denote the new tuple of $r_i$'s by $\tilde X$.

{\bf Proof of Theorem 1.} 

We will first prove the second statement. Suppose a system  of  equations in $BS(1,k)$ does not have coefficients.  Then systems (\ref{(1)}) and (\ref{(3)}) do not have the last term.

The system has a non-trivial solution if and only if system (4) has a non-trivial solution. We can describe all solutions of (4) because they come from systems of linear equations.  Then we substitute any solution of (4) in (3) and find  all solutions of (3) in $Z(1/k)$ as a homogenous system of linear equations  over $Z(1/k)$.

 It cannot happen that (4) has only finitely many solutions and the number of equations is more than the number of variables (so each $z_i=0$). Indeed then there are no  $a'$s in the solution and we have a homogeneous  linear system in abelian group that either has only zero solution or infinitely many.

Now we will  show that there is an algorithm to decide if a quadratic equation has a solution.  Every quadratic equation is equivalent to an equation in the standard form 
$$\Pi _{i=1}^g[x_i,y_i]\Pi _{i=1}^n z_i^{-1}c_ix_i=1$$ or
$$\Pi _{i=1}^{g}[x_i,y_i]\Pi _{i=1}^n z_i^{-1}c_ix_i=1$$

The commutator width  of $BS(1,n)$  is one. 
Indeed, since the relator  has $b$-exponent $0$ and $a$-exponent $1-k$, any word on $a,b$ tat represents an element of the derived subgroup must have $b$-exponent $0$ and $a$-exponent a multiple of $k-1$. Therefore, each element of the derived subgroup may be written as $b^{s}a^{m(k-1)}b^{-s},$ which is equal to
$$b^sa^{mk}b^{-s}b^sa^{-m}b^{-s}=b^{s-1}a^mb^{1-s}b^sa^{-m}b^{-s}=b^{s-1}a^mba^{-m}b^{-s}=[b, a^{-m}b^{-s}].$$

The verbal width of the subgroup generated by the squares in $BS(1,n)$ is two. Hence an orientable equation of genus $g\geq 1$ has a solution if and only if $\Pi _{i=1}^nc_i\in BS(1,n)'$. A non-orientable equation of genus $g\geq 2$ has a solution if and only if $\Pi _{i=1}^nc_i$ belongs to this verbal subgroup. Therefore (except non-orientable of genus 1) we only have to deal with equations of genus zero.
$$\Pi _{i=1}^n z_i^{-1}c_iz_i=1.$$
\begin{lemma}
The question about the existence of solutions to quadratic equation of genus zero reduces to the question about existence of solutions to certain system  (\ref{4}) or, equivalently, (\ref{eq-6}).
\end{lemma}
\begin{proof} Consider $$\Pi _{i=1}^n x_i^{-1}\bar c_ix_i=1$$ and let $x_i=(z_ik^{-y_i}, r_i)$ and $\bar c_i=(c_ik^{-d},s_i).$ Then $$\Pi _{i=1}^n(z_ik^{-y_i},r_i)(c_ik^{-d_i},s_i)(z_ik^{-y_i},r_i)^{-1}=\Pi _{i=1}^n(z_ik^{-y_i},r_i)(c_ik^{-d_i},s_i)(-z_ik^{-y_i+r_i},-r_i)$$
$$=(\sum _{i=1}^n k^{-\sum _{j=1}^{i-1}s_i}(c_ik^{-d_i-r_i}+z_ik^{-y_i}(1-k^{-s_i})), \sum _{i=1}^n s_i)=(0,0).$$
Therefore, $\sum _{i=1}^n s_i$. 

There are two possible cases. In the first case $s_i=0$ for all $i=1,\ldots ,n$, then the system is equivalent to a system $\sum _{i=1}^nc_ik^{\bar y_i}=0$ for new integer variables $\bar y_i, i=1,\ldots ,n$. This is exactly system (\ref{eq-6}).

In the second case, some $s_i$ is non-zero. Take $s=gcd(|s_1|,\ldots ,|s_i|)$, then $(k^{s}-1)=gcd((k^{|s_i|}-1), i=1,\ldots ,n)$ and the quadratic equation has a solution if and only if the  congruence
$$\sum _{i=1}^nc_ik^{\bar y_i}\equiv 0 (mod (k^{s}-1))$$ has a solution in $\mathbb Z(1/k)$. Therefore we only have to consider $\bar y_i$'s  such that $-s\leq y_i\leq s$. This finishes the proof in the second case.
\end{proof}

\section {Restricted wreath products with $\Z$}
The restricted wreath product $G \wr \Z$ is isomorphic to the semidirect product $\oplus_{i \in \Z} G \rtimes \Z$, where the action of $\Z$ on $\oplus_{i \in \Z} G$ is by translation of indices, that is, $k \cdot \{g_n\}_{n \in \Z} = \{g_{n+k}\}_{n \in \Z}$. The product of two elements $(\{g_n\}_{n \in \Z}, k) \cdot (\{h_n\}_{n \in \Z}, l)$ is $(\{g_n + h_{n+k}\}_{n \in \Z}, k+l)$. When $G= \Z_2$ the group is called the lamplighter group. 

If $A$ is finitely generated abelian, then $A=\mathbb Z^m\oplus \mathbb Z_{n_1}\oplus\ldots \oplus \mathbb Z_{n_k}$ as an additive group. Denote by $R$ the ring $\mathbb Z^m\oplus \mathbb Z_{n_1}\oplus\ldots \oplus \mathbb Z_{n_k}.$ In this case $A \wr \Z$  is isomorphic to the group of matrices of the form
\begin{equation*} 
M = \left(\begin{array}{cc} t^x & P \\ 0 & 1 \end{array}\right)
\end{equation*}
where $P$ is a Laurent polynomial in $R[t, t^{-1}]$.  Note that $P=f(t)t^{-k}$ where $f(t) \in R[t]$ and $k \in \N$. 

We will first show that equations in $A \wr \Z$ are decidable for $A=\Z_n$ and $A= \Z$. We will denote $\Z_n \wr \Z$ by $L_n$ and $\Z \wr \Z$ by $L$.
\begin{theorem} \label{th-Ln}
Quadratic equations in $L_n$  are decidable. There is also an algorithm to decide if an arbitrary coefficient free system has a non-trivial solution.
\end{theorem}

\begin{proof}
The product of $n$ elements in $L_n$ is
\begin{equation*}
\begin{pmatrix} t^{x_1} & P_1 \\ 0 & 1 \end{pmatrix} \ldots
\begin{pmatrix} t^{x_n} & P_n\\ 0 & 1 \end{pmatrix}
=
\begin{pmatrix} t^{x_1+ \ldots + x_n} & Q \\ 0 & 1 \end{pmatrix}
\end{equation*}
where $P_j= f_j(t)t^{-y_j} $ and 
\begin{equation*}
Q= f_n(t)t^{-y_n}t^{x_1 + \ldots + x_{n-1}} + f_{n-1}(t)t^{-y_{n-1}}t^{x_1 + \ldots +x_{n-2}} + \ldots +f_1(t)t^{-y_1}
\end{equation*}

In a system of equations in $L_n$, some of the $x_i, f_j(t)$ and $y_j$ may be constants and some may be variables.  

Thus, any system of equations in $L_n$ is equivalent to a system of equations of the form: 
\begin{equation} \label{eq-1}
F_1(\bar x, t, t^{-1}) f_1(t)t^{-y_1} + \ldots + F_m(\bar x, t, t^{-1})f_m(t)t^{-y_m} 
= P(\bar x, t, t^{-1}) 
\end{equation} 
and
\begin{equation} \label{eq-2}
\sum_i c_ix_i + C = 0
\end{equation}
where $F_j(\bar x, t, t^{-1})= \sum_i \alpha_i t^{\sigma_i(\bar x)} $ where $\alpha_i = \pm 1$, and $\sigma_i(\bar x)$ is a linear combination of elements in $x$ and a constant, and $f_j(t)$ is a variable  that runs over $\Z_n[t]$, $y_j$ is a variable  that runs over  $\N$, $P(\bar x, t, t^{-1})$ is a polynomial in $\Z_n[t,t^{-1}]$ with linear combinations of $\bar x$ in the exponents of $t$ and $c_i, C \in \Z$.

We begin by solving the linear system (\ref{eq-2}) as in Section 2. If the system does not have a solution, then system (\ref{eq-1}) will not have a solution either. If the system has a solution, then we substitute those values of $x_i$ into system (\ref{eq-1}). Some $x_i$ will be replaced by integers, others by linear combinations of elements in $\bar x$ and constants. 

Now we solve system (\ref{eq-1}). This system can be put in Smith normal form by regarding the terms $f_j(t)t^{-y_j}$ as variables, the terms $F_j(\bar x, t, t^{-1})$ as coefficients, and $P(\bar x, t, t^{-1})$ as a constant coefficient.

Thus, the system is equivalent to a disjunction of systems of the form:

\begin{equation} \label{eq-3}
F_s'(\bar x, t, t^{-1})f_s(t)t^{-y_s} = \sum_{i>q} F'_{s_i}(\bar x, t, t^{-1}) f_i(t)t^{-y_i} + P_s'(\bar x, t, t^{-1})
\end{equation}
for $s = 1, \ldots, q$, and
\begin{equation} \label{eq-4}
\sum_i a_it^{\sigma_i(\bar x, d_i)} =0
\end{equation}
where $a_i, \in \Z_n$ and $\sigma_i(\bar x, d_i)$ is a linear combination of elements in $\bar x$ with constants.


To solve system (\ref{eq-4}), we begin by grouping terms in each equation such that the sum of the coefficients of each group is zero modulo $n$. If there is no way to group each equation in the system in this way, then this system does not have a solution. For, suppose there is a solution to system (\ref{eq-4}), then after substituting the solution in each equation and simplifying, the coefficients of each $t^i$ should be zero in each equation, thus the sum of the coefficients of $t^i$ before simplifying must be zero modulo $n$. 

There may be many ways to group the terms of each equation. We fix one system after grouping and
for each equation, we set the powers of $t$ in the terms that were grouped together equal to each other, consequently obtaining a system of linear equations.

For example in $L_5$, the equation 
\begin{equation*}
3t^{3-x_1+x_2}+4t^{-2+x_1}+2t^{x_3-2}+1=0
\end{equation*} 
can be grouped as follows:
\begin{equation*}
(3t^{3-x_1+x_2}+2t^{x_3-2})+(4t^{-2+x_1}+1)=0
\end{equation*}

We then obtain the linear system
 \begin{equation*}
 3-x_1+x_2=x_3-2 
 \end{equation*}
 \begin{equation*}
 -2+x_1=0
 \end{equation*}
 
 We now solve this system of linear equations. If there is no solution, system (\ref{eq-3}) has no solution in this branch. If there is a solution, then we substitute the general solution back into (\ref{eq-3}).
 
 {\bf Proof of Theorem 2} 
 Every non-abelian abelian-by-cyclic group $A\rtimes _{\phi}\mathbb Z$ has commutator width 1.  Indeed,   the derived subgroup  equals the image of $\phi -1\in End(A)$ that consists of commutators. The action of $\phi$ on $A/(\phi -1)A$ is trivial, therefore $(A\rtimes _{\phi}\mathbb Z)'\in (\phi -1)A.$ Therefore, everything again reduces to genus zero equations.

 Consider $$\Pi _{i=1}^n\bar x_i\bar c_i\bar x_i^{-1}=1$$  and let $\bar x_i=\begin{pmatrix} t^{x_i} & f_i(t)t^{-y_i} \\ 0 & 1 \end{pmatrix}$ Then $\Pi _{i=1}^n\bar x_i\bar c_i\bar x_i^{-1}=\begin{pmatrix} t^{\sum _{i=1}^i s_i} & P \\ 0 & 1 \end{pmatrix}$, where

 $$P=\sum _{i=1}^n (f_i(t)t^{y_i}(1-t^{s_i})+c_it^{-d_i-x_i})t^{\sum_{j=1}^{i-1}s_i}$$

Therefore, $\sum _{i=1}^n s_i$. 

There are two possible cases. In the first case $s_i=0$ for all $i=1,\ldots ,n$, then the system is equivalent to a system $\sum _{i=1}^nc_it^{\bar y_i}=0$ for new integer variables $\bar y_i, i=1,\ldots ,n$. This is exactly system (\ref{eq-4}).

In the second case, some $s_i$ is non-zero. Take $s=gcd(|s_1|,\ldots ,|s_i|)$, then $(t^{s}-1)=gcd((t^{|s_i|}-1), i=1,\ldots ,n)$.  For non-prime $n$, the ring  $\mathbb Z_n[t,t^{-1}]$ is not a domain. But one can still use an analogue of the Euclidean algorithm and induction on $n$, to show that $t^{s}-1$ can be represented as a linear combination of $t^{|s_i|}-1, i=1,\ldots ,n$ with coefficients in $\mathbb Z_n[t,t^{-1}]$. Quadratic equation $\Pi _{i=1}^n\bar x_i\bar c_i\bar x_i^{-1}=1$ in this case has a solution if and only if the  congruence
$$\sum _{i=1}^nc_it^{\bar y_i}\equiv 0 (mod (t^{s}-1))$$ has a solution in $\mathbb Z[t,t^{-1}]$. To check this congruence we only have to consider $\bar y_i$'s  such that $-s\leq y_i\leq s$. This finishes the proof in the second case.\end{proof}
 
 The second statement of the theorem is proved similarly to the proof for $BS(1,k)$. 

\begin{theorem} \label{th-L}
Quadratic equations in $L$  are decidable. There is also an algorithm to decide if an arbitrary coefficient free system has a non-trivial solution.
\end{theorem}
A system of equations in $L$ reduces to equations of the form (\ref{eq-1}) and (\ref{eq-2}), but the $f_j(t)$ are variables in $\Z[t]$ and $P(\bar x, t, t^{-1})$ is a polynomial with coefficients in $\Z$. To solve system (\ref{eq-4}) we group terms whose coefficients add up to $0$. Then we reduce this system to system (\ref{eq-5}).

Theorem \ref{th-L} implies the following corollary.
\begin{cor}
The Diophantine problem is decidable for coefficient free and for quadratic equations in $\Z^n \wr \Z$.\end{cor}
\begin{proof}
Equations in $\Z^n \wr \Z$ have the same form as equations (\ref{eq-1}) and (\ref{eq-2}) in the proof of Theorem \ref{th-L}, with the exception that the terms $f_i(t)$ are in the ring $\Z^n[t]$. Each equation of the form (\ref{eq-1}) is equivalent to $n$ equations, each corresponding to a component of $\Z^n$. Thus, any system of equations in $\Z^n \wr \Z$ is equivalent to a system in $\Z \wr \Z$, so the decidability follows from the decidability of $\Z \wr \Z$.
\end{proof}
Combining Theorems \ref{th-Ln}  and \ref{th-L} we obtain the second main result.
\begin{theorem}
The Diophantine problem is decidable  for coefficient free and for quadratic equations in $A \wr \Z$, where $A$ is a finitely generated abelian group.
\end{theorem}
\begin{proof} Let $A=\mathbb Z^m\oplus \mathbb Z_{n_1}\oplus\ldots \oplus \mathbb Z_{n_k}$.
Equations in $A \wr \Z$ have the same form as equations (\ref{eq-1}) and (\ref{eq-2}) in the proof of Theorems \ref{th-Ln}, \ref{th-L} with the exception that the terms $f_i(t)$ are in the ring $R[t]$ (recall that $R$ is the same as $A$ but viewed as a ring). Each system of the form (\ref{eq-1}) is equivalent to several systems, some of them over $\mathbb Z$ and some over $\mathbb Z_{n_i}$, each corresponding to a component of $\mathbb Z^m\oplus \mathbb Z_{n_1}\oplus\ldots \oplus \mathbb Z_{n_k}.$ Solving these systems simultaneously we will solve the original system.\end{proof}

We conclude with some open problems.

\begin{problem} Is the Diophantine problem decidable in $BS(1,k)$ and in wreath products $A\wr \Z,$ where  $A$ is a finitely generated abelian group?
\end{problem}

\begin{problem} Is the existential theory of $BS(1,k)$ and wreath products $A\wr \Z,$ where  $A$ is a finitely generated abelian group,  decidable?
\end{problem}

\begin{problem} Describe finitely generated metabelian groups with decidable Diophantine problem.
\end{problem}

20F16, 20F70
\end{document}